\documentclass[a4paper,12pt]{amsart}

\usepackage{a4wide,url}
\usepackage{tikz}
\usetikzlibrary{decorations.pathreplacing}
\usetikzlibrary{shapes}

\newcommand{\T}{\mathcal{T}}
\newcommand{\D}{\mathcal{D}}

\newcommand{\EE}{\mathbb{E}}

\newif\ifdetails
\detailstrue
\newcommand{\DETAIL}[1]%
{\ifdetails\par\fbox{\begin{minipage}{0.9\linewidth}\textit{Detail:}
      #1\end{minipage}}\par\fi}
\newcommand{\TODO}[1]%
{\ifdetails\par\fbox{\begin{minipage}{0.9\linewidth}\textbf{TODO:}
      #1\end{minipage}}\par\fi}

\newtheorem{lemma}{Lemma}

\newtheorem{theorem}[lemma]{Theorem}

\newtheorem{claim}{Claim}

\newcommand{\Crt}{\operatorname{crt}}
\newcommand{\Cr}{\operatorname{cr}}
\newcommand{\lca}{\operatorname{lca}}
\newcommand{\inter}{\operatorname{int}}

\newcommand{\old}[1]{{}}
\makeatletter
\DeclareRobustCommand{\cev}[1]{%
  {\mathpalette\do@cev{#1}}%
}
\newcommand{\do@cev}[2]{%
  \vbox{\offinterlineskip
    \sbox\z@{$\m@th#1 x$}%
    \ialign{##\cr
      \hidewidth\reflectbox{$\m@th#1\vec{}\mkern4mu$}\hidewidth\cr
      \noalign{\kern-\ht\z@}
      $\m@th#1#2$\cr
    }%
  }%
}
\makeatother

\title{The largest crossing number of tanglegrams}

\author{\'Eva Czabarka, Junsheng Liu, L\'aszl\'o A. Sz\'ekely }
\address{\'Eva Czabarka Junsheng Liu, and L\'aszl\'o A. Sz\'ekely\\ Department of Mathematics \\ University of South Carolina \\ Columbia, SC 29208 \\ USA}
\email{\{czabarka,szekely\}@math.sc.edu, junsheng@email.sc.edu}

\begin{document}

\begin{abstract}
A tanglegram $\T$ consists of two rooted binary   trees with the same number of leaves,  and a perfect matching between the two leaf sets. In a layout, the tanglegrams is drawn with the leaves on two parallel lines, the trees on either side of the strip created by these lines are drawn as plane trees, and the perfect matching is drawn in straight line segments inside the strip.  The tanglegram crossing number   $\Crt(\T)$ of $\T$ is the smallest number of crossings of pairs of matching edges, over 
all possible layouts of $\T$.  The size of the tanglegram is the number of matching edges, say $n$.
An earlier paper showed that the maximum  of the tanglegram crossing number of size $n$ tanglegrams is $<\frac{1}{2}\binom{n}{2}$; but is at least $\frac{1}{2}\binom{n}{2}-\frac{n^{3/2}-n}{2}$ for infinitely many $n$.
Now we make better bounds:  the maximum crossing number of a size $n$ tanglegram is at most
$ \frac{1}{2}\binom{n}{2}-\frac{n}{4}$, but for infinitely many $n$, at least
$\frac{1}{2}\binom{n}{2}-\frac{n\log_2 n}{4}$.  The problem shows analogy with the Unbalancing Lights Problem
of Gale and Berlekamp.
\end{abstract}

\maketitle

\section{introduction} \label{intro}
A {\em  binary tree} has a root vertex
assumed to be a common ancestor of all other vertices, and each vertex either has two children
 or no children. A vertex with no children is a  {\em leaf}, and a vertex with two
children is an {\em internal vertex}. Note that this definition allows a single-vertex tree that is considered as both root and leaf to be a rooted binary tree.
In an {\em ordered binary tree} an order of the two children is specified,  for every vertex that has children.

A {\em  plane binary tree} is a drawn ordered binary  tree, without edge crossings, where the left-right order of subtrees in the drawing coincides with the order. The edges are drawn in straight line segments.
It is easy to draw a plane binary tree  in such a way that all the leaves are on a line, and all other vertices are in the same open halfplane.

A  {\em tanglegram} $\T=({L},{R},\sigma)$ is a graph that consists of a left  binary tree ${L}$, 
a right  binary tree ${R}$ with the same number of leaves as $L$, 
and a perfect matching  $\sigma$ between the leaves of ${L}$ and ${R}$. Two tanglegrams are considered identical, if there is a graph isomorphism between
them fixing the root $r$ of $R$ and the root $\rho$ of $L$. The {\em size} of a tanglegram is the number of leaves in $L$ (or $R$).
An  {\em abstract tanglegram layout} of the tanglegram    $(L,R,\sigma)$       is given by turning the unordered trees   $L$ and $R$ into {\em ordered trees}. Given an abstract tanglegram layout, an actual {\em tanglegram layout}  consists of a left plane binary tree isomorphic     (keeping order as well) to $L$   with root $r$ drawn in the halfplane $x\leq 0$, having
its leaves on the  line $x=0$,
a right plane binary tree isomorphic  (keeping order as well)  to $R$  with root $\rho$, drawn in the halfplane $x\geq 1$, having
its leaves on the  line $x=1$,    
and a perfect matching  $\sigma$ between their leaves drawn in straight line segments. (Isomorphism of ordered trees (plane trees)
keeps the root and the order.)

Our main concern about tanglegram layouts is the number of crossings between the matching edges. As it is determined by the abstract tanglegram layout, it is  sufficient to focus on the abstract tanglegram layout to count crossings. 

A {\em switch} on the abstract  tanglegram layout $(L,R,\sigma)$ is the following operation: select an internal vertex $v$ of one of the two trees $L$ and $R$ and change the order of its two children. 

It is easy to see that two abstract  tanglegram layouts represent the same tanglegram if and only if a sequence of switches  moves one abstract layout into the other. (A switch is on a tanglegram layout
 illustrated in Figure~\ref{fig:sw&m}.) Hence tanglegrams of a given size partition the set of all abstract tanglegram layouts of the same size, or equivalently a tanglegram can be seen as an equivalence class of abstract tanglegram layouts.
Note that interchanging $L$ and $R$ is not allowed, as it may result in a different tanglegram.

\begin{figure}[htbp]
\begin{center}
\begin{tikzpicture}[scale=.67]
        \node[fill=black,circle,inner sep=1.5pt]  at (-1,0) {}; 
        \node[fill=black,circle,inner sep=1.5pt]  at (1,-1) {};
        \node[fill=black,circle,inner sep=1.5pt]  at (1,1) {};
        \node[fill=black,rectangle,inner sep=2pt]  at (2,-1.5) {};
        \node[fill=black,rectangle,inner sep=2pt]  at (2,-.5) {};
        \node[fill=black,rectangle,inner sep=2pt]  at (2,.5) {};
        \node[fill=black,rectangle,inner sep=2pt]  at (2,1.5) {};

	\draw (-1,0)--(2,1.5);
	\draw (-1,0)--(2,-1.5);
	\draw (1,-1)--(2,-.5);
	\draw (1,1)--(2,.5);

        \node[fill=black,rectangle,inner sep=2pt]  at (3,-1.5) {};
        \node[fill=black,rectangle,inner sep=2pt]  at (3,-.5) {};
        \node[fill=black,rectangle,inner sep=2pt]  at (3,.5) {};
        \node[fill=black,rectangle,inner sep=2pt]  at (3,1.5) {};
        \node[fill=black,circle,inner sep=1.5pt]  at (4,1) {};
        \node[fill=black,circle,inner sep=1.5pt]  at (5,.5) {};
        \node[fill=black,circle,inner sep=1.5pt]  at (6,0) {}; 

	\draw (6,0)--(3,1.5);
	\draw (6,0)--(3,-1.5);
	\draw (5,.5)--(3,-.5);
	\draw (4,1)--(3,.5);

	\draw [dashed] (2,-1.5)--(3,-1.5);
	\draw [dashed] (2,-.5)--(3,.5);
	\draw [dashed] (3,-.5)--(2,.5);
	\draw [dashed] (2,1.5)--(3,1.5);
	
	\node at (2,-1.9) {$a$};
	\node at (3,-1.9) {$a$};
         \node at (2,-.9)  {$b$};
         \node at (3,.1)  {$b$};
         \node at (3,-.9)  {$c$};
         \node at (2,.1)  {$c$};
         \node at  (2,1.1) {$d$};
         \node at (3,1.1)    {$d$};
	\node at (2.5,-2.5) {original layout};
	\node at (-1,0.3) {$r$};
	\node at (6,0.3) {$\rho$};
         
        \node[fill=black,circle,inner sep=1.5pt]  at (7,0) {}; 
        \node[fill=black,circle,inner sep=1.5pt]  at (9,-1) {};
        \node[fill=black,circle,inner sep=1.5pt]  at (9,1) {};
        \node[fill=black,rectangle,inner sep=2pt]  at (10,-1.5) {};
        \node[fill=black,rectangle,inner sep=2pt]  at (10,-.5) {};
        \node[fill=black,rectangle,inner sep=2pt]  at (10,.5) {};
        \node[fill=black,rectangle,inner sep=2pt]  at (10,1.5) {};

	\draw (7,0)--(10,1.5);
	\draw (7,0)--(10,-1.5);
	\draw (9,-1)--(10,-.5);
	\draw (9,1)--(10,.5);

        \node[fill=black,rectangle,inner sep=2pt]  at (11,-1.5) {};
        \node[fill=black,rectangle,inner sep=2pt]  at (11,-.5) {};
        \node[fill=black,rectangle,inner sep=2pt]  at (11,.5) {};
        \node[fill=black,rectangle,inner sep=2pt]  at (11,1.5) {};
        \node[fill=black,circle,inner sep=1.5pt]  at (12,0) {};
        \node[fill=black,circle,inner sep=1.5pt]  at (13,-.5) {};
        \node[fill=black,circle,inner sep=1.5pt]  at (14,0) {}; 

	\draw (14,0)--(11,1.5);
	\draw (14,0)--(11,-1.5);
	\draw (13,-.5)--(11,.5);
	\draw (12,0)--(11,-.5);

	\draw [dashed] (10,-1.5)--(11,1.5);
	\draw [dashed] (10,-.5)--(11,-.5);
	\draw [dashed] (10,.5)--(11,-1.5);
	\draw [dashed] (10,1.5)--(11,.5);
	
	\node at (10,-1.9) {$a$};
	\node at (10,-.9) {$b$};
	\node at (10,.1) {$c$};
	\node at (10,1.1) {$d$};
	\node at (11,1.1) {$a$};
	\node at (11,0.1) {$d$};
	\node at (11,-.9) {$b$};
	\node at (11,-1.9) {$c$};
	\node at (10.5,-2.5) {after switching at $\rho$};
	\node at (7,0.3) {$r$};
	\node at (14,0.3) {$\rho$};

\old{
        \node[fill=black,circle,inner sep=1.5pt]  at (15,0) {};
        \node[fill=black,circle,inner sep=1.5pt]  at (17,-1) {};
        \node[fill=black,circle,inner sep=1.5pt]  at (17,1) {};
        \node[fill=black,rectangle,inner sep=2pt]  at (18,-1.5) {};
        \node[fill=black,rectangle,inner sep=2pt]  at (18,-.5) {};
        \node[fill=black,rectangle,inner sep=2pt]  at (18,.5) {};
        \node[fill=black,rectangle,inner sep=2pt]  at (18,1.5) {};
	\draw (15,0)--(18,1.5);
	\draw (15,0)--(18,-1.5);
	\draw (17,-1)--(18,-.5);
	\draw (17,1)--(18,.5);
        \node[fill=black,rectangle,inner sep=2pt]  at (19,-1.5) {};
        \node[fill=black,rectangle,inner sep=2pt]  at (19,-.5) {};
        \node[fill=black,rectangle,inner sep=2pt]  at (19,.5) {};
        \node[fill=black,rectangle,inner sep=2pt]  at (19,1.5) {};
        \node[fill=black,circle,inner sep=1.5pt]  at (20,1) {};
        \node[fill=black,circle,inner sep=1.5pt]  at (21,.5) {};
        \node[fill=black,circle,inner sep=1.5pt]  at (22,0) {};
	\draw (22,0)--(19,1.5);
	\draw (22,0)--(19,-1.5);
	\draw (21,.5)--(19,-.5);
	\draw (20,1)--(19,.5);
	\draw [dashed] (18,-1.5)--(19,1.5); 
	\draw [dashed] (18,-.5)--(19,-.5); 
	\draw [dashed] (19,.5)--(18,.5);
	\draw [dashed] (18,1.5)--(19,-1.5);
	\node at (18,-1.9) {$d$};
	\node at (19,-1.9) {$a$};
         \node at (19,-.9)  {$c$};
         \node at (19,.1)  {$b$};
         \node at (18,-.9)  {$c$};
         \node at (18,.1)  {$b$};
         \node at  (18,1.1) {$a$};
         \node at (19,1.1)    {$d$};
         \node at (18.5,-2.5) {after mirroring at $r$};
         \node at (15,0.3) {$r$};
	\node at (22,0.3) {$\rho$};
  }        
        
\end{tikzpicture}
\end{center}
\caption{Result of a switch 
operation. 
}  \label{fig:sw&m}
\end{figure}
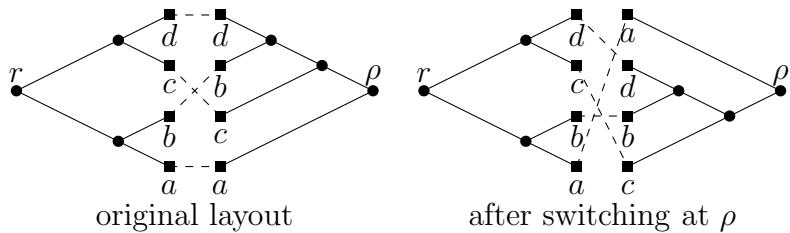

The crossing number of a tanglegram layout is the number of  pairs of matching edges that cross, which is determined by the 
abstract tanglegram layout.

 It is desirable to draw a tanglegram with the {\em least possible number of crossings}, which is known as the 
 Tanglegram Layout Problem \cite{fernau, St.John}. The {\em (tanglegram) crossing number}  $\Crt(\T)$ of a tanglegram $\T$  is defined as the minimum number of crossings among its layouts. 
The  Tanglegram Layout Problem  problem is NP-hard \cite{buchin,fernau}, but is Fixed Parameter Tractable \cite{buchin, bocker}. It does not allow constant factor approximation under the Unique Game Conjecture \cite{buchin}. 
Tanglegrams play a major role in phylogenetics, especially in the theory of cospeciation \cite{page}. 
For example, the first binary tree is the phylogenetic tree of the hosts, the second binary tree is the phylogenetic tree of their parasites (e.g., gopher and louse), and the matching connects the host with its parasite
\cite{HafnerNadler}. 
 The tanglegram  crossing number  has been related to the  number of times  parasites switched hosts \cite{HafnerNadler},
or, working with gene trees instead of phylogenetic trees, to the 
number of horizontal gene transfers (\cite{Burt}, pp. 204--206).
Tanglegrams are well-studied objects in phylogenetics and computer science. 

Let $M_n$ denote $\max_\T \Crt(\T)$ among size $n$ tanglegrams.
It is easy to see that for any tanglegram, the expected number of crossing in random layout of any fixed labeled tangegram of size $n$ is  $\frac{1}{2}\binom{n}{2}$. (For details, see Section~\ref{switches}.)
Therefore, $M _n \leq \frac{1}{2}\binom{n}{2}$. 
An earlier paper \cite{MRCtangle} made a slight improvent showing that equality cannot happen:  $M_n<\frac{1}{2}\binom{n}{2}$; 
and also showed that for every $n=2^k$,
$$\frac{1}{2}\binom{n}{2}-\frac{n^{3/2}-n}{2}\leq M_n.$$
The goal of this paper is to find more  proper separation as $M_n\leq \frac{1}{2}\binom{n}{2}-\frac{n}{4} ,$ and for infinitely many $n$,
$$M_n\geq \frac{1}{2}\binom{n}{2}-\frac{n\log_2 n}{4}.$$

In Section~\ref{evaexample} we provide a construction for the lower bound, 
in Section~\ref{switches} we relate the number of crossings in different layouts of the tanglegram.
In Section~\ref{unbalance} we relate the largest crossing number problem  
to the  Unbalancing Lights Problem
of Gale and Berlekamp, and show the separation from $\frac{1}{2}\binom{n}{2}$.  In Section~\ref{tools} we derive some technical results that we need for the proof.

\section{A construction for  tanglegrams with large crossing number}
\begin{theorem} \label{evaexample}
For every $i\geq 1$,
there exists a tanglegram of size $2^i$, which has tanglegram crossing number $\frac{1}{2}\binom{2^i}{2}-i2^{i-2}$ exactly.
\end{theorem}

Let $X=\{0,1\}$ and let $X^i$ be the set of binary strings  of length $i$, i.e., words over the alphabet $X$, which make the binary representations of the non-negative integers that are less than $2^i$.
Given a string  $\vec{x}=x_1x_2\ldots x_i$,  we will denote by $\cev{x}$ the string obtained by reversing $x$, i.e.,
$\cev{x}=x_ix_{i-1}\ldots x_1$. 

For every $i\in\mathbb{N}$ we will define a   tanglegram $\mathcal{T}_i=(R^{(i)},L^{(i)},\sigma_i)$ of size $2^i$  by the following procedure:

Both $L^{(i)}$ and $R^{(i)}$ are the rooted complete binary trees of height $i$. We label the vertices of $L^{(i)}$ (resp. of $R^{(i)}$) as follows: The set of vertices at distance $j$ (which we call the $j^{th}$ layer)
from the root
are labeled as $u_{\vec{x}}$ (resp. $w_{\vec{x}}$) where $\vec{x}$ is an element of $X^j$. The root  of  $L^{(i)}$)  is labeled as $u_{\epsilon}$, and the root of $R^{(i)}$)  is labeled as $v_{\epsilon}$, where $\epsilon$ is the empty string.
The labels of children of $u_{\vec{x}}$ (resp. $w_{\vec{x}}$) are created by suffices:   $u_{\vec{x}0}$ and $u_{\vec{x}1}$ (resp. $w_{\vec{x}0}$ and $w_{\vec{x}1}$). The matching is 
$\sigma_i=\{u_{\vec{x}}w_{\cev{x}}:\vec{x}\in X^{i}\}$, see Fig.~\ref{fig:Evatangs}.

For $t\in X$, let $L^{(i)}_{t}$ (resp. $R^{(i)}_{t}$) denote the subtree of $L^{(i)}$ (resp. $R^{(i)}$) rooted at $u_{t}$ (resp. $w_{t}$).

\begin{figure}[ht]
		\centering
			\begin{tikzpicture}
			[scale=1,xscale=1,inner sep=1pt,semithick,font=\tiny,
			vertex/.style={rounded rectangle,draw,minimum height=3mm},
			thickedge/.style={line width=0.75pt},
			plumpedge/.style={gray!60,line width=4pt},
			] 
\node[vertex] (x000) at (-7,2.1) {};
\node[vertex] (y000) at (-6,2.1) {};
\draw[thickedge] (x000)--(y000);
\node[vertex] (x010) at (-7,1.2) {0};
\node[vertex] (x011) at (-7,0.6) {1};
\node[vertex] (y010) at (-6,1.2) {0};
\node[vertex] (y011) at (-6,0.6) {1};
\node[vertex] (x01) at (-7.5,0.9) {};
\node[vertex] (y01) at (-5.5,0.9) {};
\draw[thickedge] (x010)--(y010);
\draw[thickedge] (x011)--(y011);
\node[vertex] (x100) at (-7,-0.3) {00};
\node[vertex] (x101) at (-7,-0.9) {01};
\node[vertex] (x110) at (-7,-1.5) {10};
\node[vertex] (x111) at (-7,-2.1) {11};
\node[vertex] (y100) at (-6,-0.3) {00};
\node[vertex] (y101) at (-6,-0.9) {01};
\node[vertex] (y110) at (-6,-1.5) {10};
\node[vertex] (y111) at (-6,-2.1) {11};
\node[vertex] (x10) at (-7.5,-0.6) {0};
\node[vertex] (x11) at (-7.5,-1.8) {1};
\node[vertex] (y10) at (-5.5,-0.6) {0};
\node[vertex] (y11) at (-5.5,-1.8) {1};
\node[vertex] (x1) at (-8.5,-1.2) {};
\node[vertex] (y1) at (-4.5,-1.2) {};
\draw[thickedge] (x1)--(x11)--(x111);
\draw[thickedge] (y111)--(y11)--(y1);
\draw[thickedge] (x11)--(x110);
\draw[thickedge] (y11)--(y110);
\draw[thickedge] (x01)--(x010);
\draw[thickedge] (y01)--(y010);
\draw[thickedge] (x1)--(x10)--(x101);
\draw[thickedge] (y1)--(y10)--(y101);
\draw[thickedge] (x01)--(x011);
\draw[thickedge] (y01)--(y011);
\draw[thickedge] (x10)--(x100);
\draw[thickedge] (y10)--(y100);
\draw[thickedge] (x100)--(y100);
\draw[thickedge] (x101)--(y110);
\draw[thickedge] (x110)--(y101);
\draw[thickedge] (x111)--(y111);
\node[vertex] (u000) at (0,2.1) {000};
\node[vertex] (u001) at (0,1.5) {001};
\node[vertex] (u010) at (0,0.9) {010};
\node[vertex] (u011) at (0,0.3) {011};
\node[vertex] (u100) at (0,-0.3) {100};
\node[vertex] (u101) at (0,-0.9) {101};
\node[vertex] (u110) at (0,-1.5) {110};
\node[vertex] (u111) at (0,-2.1) {111};
\node[vertex] (w000) at (1.5,2.1) {000};
\node[vertex] (w001) at (1.5,1.5) {001};
\node[vertex] (w010) at (1.5,0.9) {010};
\node[vertex] (w011) at (1.5,0.3) {011};
\node[vertex] (w100) at (1.5,-0.3) {100};
\node[vertex] (w101) at (1.5,-0.9) {101};
\node[vertex] (w110) at (1.5,-1.5) {110};
\node[vertex] (w111) at (1.5,-2.1) {111};
\node[vertex] (u00) at (-.5,1.8) {00};
\node[vertex] (u01) at (-.5,0.6) {01};
\node[vertex] (u10) at (-0.5,-0.6) {10};
\node[vertex] (u11) at (-0.5,-1.8) {11};
\node[vertex] (w00) at (2,1.8) {00};
\node[vertex] (w01) at (2,0.6) {01};
\node[vertex] (w10) at (2,-0.6) {10};
\node[vertex] (w11) at (2,-1.8) {11};
\node[vertex] (u0) at (-1.5,1.2) {0};
\node[vertex] (u1) at (-1.5,-1.2) {1};
\node[vertex] (w0) at (3,1.2) {0};
\node[vertex] (w1) at (3,-1.2) {1};
\node[vertex] (ue) at (-3.5,0) {};
\node[vertex] (we) at (5,0) {};
\draw[thickedge] (u000)--(u00)--(u0)--(ue)--(u1)--(u11)--(u111);
\draw[thickedge] (w111)--(w11)--(w1)--(we)--(w0)--(w00)--(w000);
\draw[thickedge] (u00)--(u001);
\draw[thickedge] (w00)--(w001);
\draw[thickedge] (u11)--(u110);
\draw[thickedge] (w11)--(w110);
\draw[thickedge] (u0)--(u01)--(u010);
\draw[thickedge] (w0)--(w01)--(w010);
\draw[thickedge] (u1)--(u10)--(u101);
\draw[thickedge] (w1)--(w10)--(w101);
\draw[thickedge] (u01)--(u011);
\draw[thickedge] (w01)--(w011);
\draw[thickedge] (u10)--(u100);
\draw[thickedge] (w10)--(w100);
\draw[thickedge] (u000)--(w000);
\draw[thickedge] (u001)--(w100);
\draw[thickedge] (u010)--(w010);
\draw[thickedge] (u011)--(w110);
\draw[thickedge] (u100)--(w001);
\draw[thickedge] (u101)--(w101);
\draw[thickedge] (u110)--(w011);
\draw[thickedge] (u111)--(w111);
\end{tikzpicture}
\caption{The tanglegrams $\T_i$ for $i\in\{0,1,2,3\}$. The vertices are labeled with their indices as in the text and the tanglegrams are shown with
a crossing-optimal layout.}
\label{fig:Evatangs}
\end{figure}
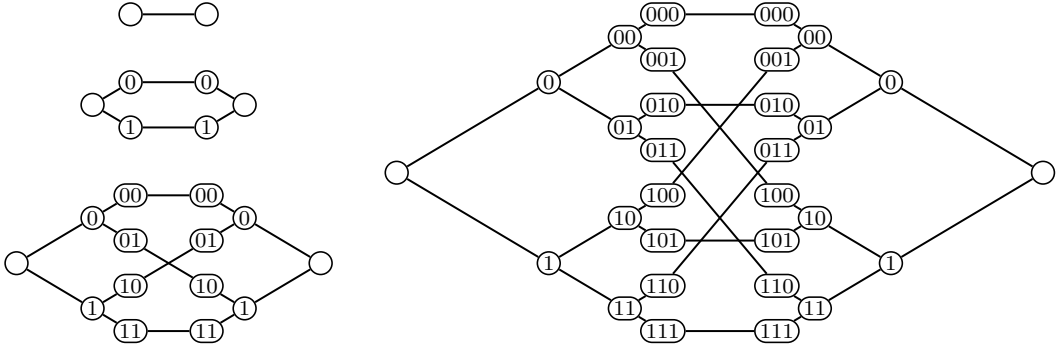

\begin{claim}\label{cl:lower}
For every $i\ge 2$, $\Crt(\mathcal{T}_i)\ge2\Crt(\mathcal{T}_{i-1})+\binom{2^{i-1}}{2}$.
\end{claim}

\begin{proof}
Let $\mathcal{T}=(L,R,\sigma)$  be an arbitrary tanglegram, and  $v$ be a non-leaf vertex of one of the trees $L,R$. Let  $Z$ be the set of leaves in the tree, where $v$ lives ($L$ or $R$) that are descendants of $v$.
Note that in any layout of $\mathcal{T}$, the elements of $Z$ appear consecutively in the sequence of leaves. Moveover, if both children of $v$ are leaves,
and the matching edges incident upon these children cross in the layout, then switching the order of these children in the layout eliminates this crossing and decreases the crossing number, so the original layout was not optimal.

Assume $i\ge 2$, and let $\mathcal{D}$ be an optimal layout of $\mathcal{T}_i$.  
Let $A$ be the number of crossings in $\mathcal{D}$ between edges incident upon leaves  $u_{\vec{x}}$ and $u_{\vec{y}}$ where  the first digits of $\vec{x}$ and $\vec{y}$ are the same 
and let $B$ be the the number of crossings in $\mathcal{D}$ between edges incident upon leaves $u_{\vec{x}}$ and $u_{\vec{y}}$ where the first digit of $\vec{x}$ and $\vec{y}$ differ. 
Obviously, $\Crt(\mathcal{T}_i)=A+B$. 
As for each $t\in X$ the matching edges incident upon leaves of $L^{(i)}_t$ induce a  $\mathcal{T}_{i-1}$ with a sublayout in $\mathcal{D}$ with at least $\Crt(\mathcal{T}_{i-1})$ crossings,  
$A\ge 2\Crt(\mathcal{T}_{i-1})$, so it is enough to show that $B\ge \binom{2^{i-1}}{2}$.

Let $t,s\in X$ be chosen such that $u_{t}$ lies above $u_{s}$ in the layout $\mathcal{D}$. Let $\vec{x},\vec{y}\in X^{i-1}$ be different words. Clearly,
$w_{\vec{x}0}, w_{\vec{x}1},w_{\vec{y}0},w_{\vec{y}1}$ are distinct leaves of $R^{(i)}$, and
$u_{0\cev{x}},u_{1\cev{x}},u_{0\cev{y}},u_{1\cev{y}}$ are distinct leaves of $L^{(i)}$. Also, the leaves $w_{\vec{x}0}, w_{\vec{x}1}$ as well as $w_{\vec{y}0},w_{\vec{y}1}$ 
are consecutive in any layout, including $D$.

We may assume without loss of generality that $u_{t\cev{x}}$ is above $u_{t\cev{y}}$ in $\mathcal{D}$.  As $u_{t}$ lies above $u_{s}$, 
both $u_{s\cev{x}},u_{s\cev{y}}$ lie below $u_{t\cev{y}}$.
If the pair  $w_{\vec{x}0}, w_{\vec{x}1}$ lies above the pair $w_{\vec{y}0},w_{\vec{y}1}$,  then the matching edges
incident upon $u_{s\cev{y}}$ and $u_{t\cev{x}}$ cross; otherwise the matching edges incident upon $u_{t\cev{y}}$ and $u_{s\cev{x}}$ cross.
 This shows that for any $\vec{x},\vec{y}\in X^{i-1}$, if $\vec{x}\ne\vec{y}$, then for some $k,\ell$ such that
$\{k,\ell\}=\{0,1\}$ we have that the matching edges incident upon $u_{k\vec{x}}$ and $u_{\ell\vec{y}}$ cross in $\mathcal{D}$. Therefore we have  $B\ge\binom{2^{i-1}}{2}$.
\end{proof}

\begin{claim}
Let $\mathcal{D}_i^{\star}$ be the layout of $\mathcal{T}_i$, in which the leaf labels from top to bottom appear in the order of the integers corresponding to the binary words,  both in $L^{(i)}$ and $R^{(i)}$. (See {\rm Fig.~\ref{fig:Evatangs}} for this layout.)
Let $\Cr(\mathcal{D}_i^{\star})$ denote the number of crossings in this layout.
Then,  for all $i\in\mathbb{N}$, we have 
$$\Cr(\mathcal{D}_i^{\star})=\Crt(\mathcal{T}_i)=\frac{1}{2}\binom{2^i}{2}-i2^{i-2}.$$
\end{claim}

\begin{proof} 
Set $\omega_i=\Cr(\mathcal{D}_i^{\star})$. We will show the statement by induction on $i$, with base cases $i\in\{0,1\}$.

$\T_0$ and $\T_1$ are the unique planar tanglegrams of size $1$ and $2$ respectively, 
$\frac{1}{2}\binom{2^0}{2}-0\cdot 2^{0-2}=0=\Crt(\mathcal{T}_0)$ and $\frac{1}{2}\binom{2^1}{2}-1\cdot 2^{1-2}=0=\Crt(\mathcal{T}_1)$.
Since for $i\in\{0,1\}$ we have $\vec{x}=\cev{x}$ for any $\vec{x}\in X^i$, $\mathcal{D}_i^{\star}$ is a planar layout, so $\omega_0=\omega_1=0$.
Thus, the statement is true
for $i\in\{0,1\}$.

Assume now $i>1$, and consider the layout $\mathcal{D}_i^{\star}$.  For each $t\in X$, the matching edges incident upon a leaf of $L^{i}_t$ 
induce a drawing of a subtanglegram of $\mathcal{T}_{i-1}$ that is isomorphic to $\mathcal{D}_{i-1}^{\star}$, contributing exactly $2\omega_{i-1}$ crossings.
We want to count the number of crossings in $\mathcal{D}_i^{\star}$ between matching edges whose left-endpoints are $u_{0\cev{x}},u_{1\cev{y}}$, where $\vec{x},\vec{y}\in X^{i-1}$.
The edges cross precisely when $\vec{y}1<\vec{x}0$, which is equivalent with $\vec{y}<\vec{x}$  (where we consider the words as binary representations of numbers). So we have exactly one such crossings
for each unordered pair $\vec{x},\vec{y}$ from $X^{i-1}$.
By the induction hypothesis and Claim~\ref{cl:lower} we have
$$\Crt(\mathcal{T}_i)\le \omega_i=2\omega_{i-1}+\binom{2^{i-1}}{2}=2\Crt(\mathcal{T}_{i-1})+\binom{2^{i-1}}{2}\le \Crt(\mathcal{T}_i),$$
which gives $\Crt(\mathcal{T}_i)=\omega_i$.
Also, by the induction hypothesis
\begin{eqnarray*}
\omega_i&=&2\omega_{i-1}+\binom{2^{i-1}}{2}=2\left(\frac{1}{2}\binom{2^{i-1}}{2}-(i-1)2^{i-3}\right)+\binom{2^{i-1}}{2}
\\
&=&2^{i-1}(2^{i-1}-1)-(i-1)2^{i-2}=\frac{1}{2}\binom{2^i}{2}-i2^{i-2}.
\end{eqnarray*}
\end{proof}
Unfortunately, we know that this construction is not the best possible. For size 8, the tanglegram on Fig.~\ref{fig:Junsheng}
is shown with an optimal drawing and has one more crossings than our construction on  Fig.~\ref{fig:Evatangs}.
\begin{figure}[ht]
		\centering
			\begin{tikzpicture}
			[scale=1,xscale=1,inner sep=2pt,semithick,
			vertex/.style={fill=black,rectangle,inner sep=2pt},
			vertexb/.style={fill=black,circle,inner sep=1.5pt]},
			thickedge/.style={line width=0.75pt},
			plumpedge/.style={gray!60,line width=4pt}] 
\node[vertex] (u000) at (0,1.75) {};
\node[vertex] (u001) at (0,1.25) {};
\node[vertex] (u010) at (0,.75) {};
\node[vertex] (u011) at (0,0.25) {};
\node[vertex] (u100) at (0,-0.25) {};
\node[vertex] (u101) at (0,-.75) {};
\node[vertex] (u110) at (0,-1.25) {};
\node[vertex] (u111) at (0,-1.75) {};
\node[vertex] (w000) at (1,1.75) {};
\node[vertex] (w001) at (1,1.25) {};
\node[vertex] (w010) at (1,.75) {};
\node[vertex] (w011) at (1,0.25) {};
\node[vertex] (w100) at (1,-0.25) {};
\node[vertex] (w101) at (1,-.75) {};
\node[vertex] (w110) at (1,-1.25) {};
\node[vertex] (w111) at (1,-1.75) {};
\node[vertexb] (u00) at (-.5,1.5) {};
\node[vertexb] (u01) at (-.5,.5) {};
\node[vertexb] (u10) at (-0.5,-.5) {};
\node[vertexb] (u11) at (-0.5,-1.5) {};
\node[vertexb] (w00) at (1.5,1.5) {};
\node[vertexb] (w01) at (1.5,.5) {};
\node[vertexb] (w10) at (1.5,-.5) {};
\node[vertexb] (w11) at (1.5,-1.5) {};
\node[vertexb] (u0) at (-1.5,1) {};
\node[vertexb] (u1) at (-1.5,-1) {};
\node[vertexb] (w0) at (2.5,1) {};
\node[vertexb] (w1) at (2.5,-1) {};
\node[vertexb] (ue) at (-3.5,0) {};
\node[vertexb] (we) at (4.5,0) {};
\draw[thickedge] (u000)--(u00)--(u0)--(ue)--(u1)--(u11)--(u111);
\draw[thickedge] (w111)--(w11)--(w1)--(we)--(w0)--(w00)--(w000);
\draw[thickedge] (u00)--(u001);
\draw[thickedge] (w00)--(w001);
\draw[thickedge] (u11)--(u110);
\draw[thickedge] (w11)--(w110);
\draw[thickedge] (u0)--(u01)--(u010);
\draw[thickedge] (w0)--(w01)--(w010);
\draw[thickedge] (u1)--(u10)--(u101);
\draw[thickedge] (w1)--(w10)--(w101);
\draw[thickedge] (u01)--(u011);
\draw[thickedge] (w01)--(w011);
\draw[thickedge] (u10)--(u100);
\draw[thickedge] (w10)--(w100);
\draw[thickedge] (u000)--(w000);
\draw[thickedge] (u001)--(w100);
\draw[thickedge] (u010)--(w010);
\draw[thickedge] (u011)--(w110);
\draw[thickedge] (u100)--(w011);
\draw[thickedge] (u101)--(w101);
\draw[thickedge] (u110)--(w001);
\draw[thickedge] (u111)--(w111);

\end{tikzpicture}
\caption{A tanglegram of size $8$ with tanglegram crossing number $9$. This is the maximum tanglegram crossing number for size 8, found by brute force search.}
\label{fig:Junsheng}
\end{figure}
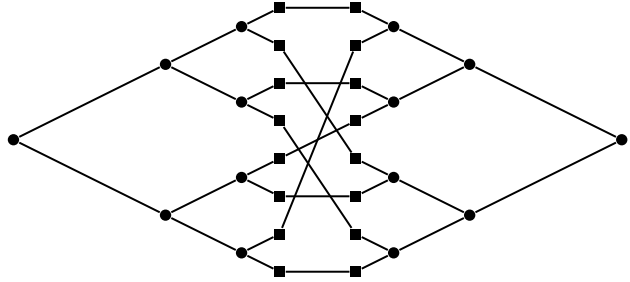

\section{Crossings in different layouts of the same tanglegram} \label{switches}
Let us be given a tanglegram $\T=(R,L,\sigma)$ of size $n$. Vertices of $R$ make a partially ordered set, for the following order: if $r$ is the root of $R$, then $x\leq y$, if $y$ is a vertex of the unique $rx$ path in $\T$.
This partial order is a semilattice, in which the least upper bound of vertices $u$ and $v$ is denoted by $\lca_R(u,v)$ ($\lca$ stands for least common ancestor in phylogenetics). Similar arguments apply for the
tree $L$, where the notation will be $\lca_L$. For $e,f$ matching edges, $\lca_R(e,f)$ (resp. $\lca_L(e,f)$) will denote the $\lca$ of the two leaves of $R$ adjacent to the edges $e$ and $f$ (resp. the $\lca$ of
the two leaves of $L$ adjacent to the edges $e$ and $f$). 

Consider a layout $\D_0$ of $\T$.  Assume that a layout $\D$ is obtained from $\D_0$ by making a switch in certain internal (non-leaf) vertices of $R$ and certain internal (non-leaf) vertices of $L$.  Note that changing the order of switches 
make no effect on $\D$.
Also note that each of $R$, $L$ has
exactly $n-1$ internal vertices. We denote the set of internal vertices by $\inter(R)$ and $\inter(L)$. 

 Define $\alpha$ ($\beta$)  on  $\inter(R)$  ($\inter(L)$) as 1, if no switch takes place in the vertex, and $-1$, if a switch takes place in the vertex. Fixing $\D_0$,  the
combinatorially different layouts $\D$ are in one-to-one correspondence with $(\alpha,\beta)$ pairs of $\pm 1$ valued functions.

Consider now two matching edges, $e,f$ of $\T$. Let $x=  \lca_R(e,f)$ and $u=\lca_L(e,f)$. Define the {\it crossing status} of matching 
edges $e,f$ in layout $\D$ as 
$$\chi_\D(e,f)= 
\begin{cases} -1  &  {\ \rm{if}\ } e {\rm \ crosses \ } f  \rm{\  in\ } \D; 
                         \cr 
                        1  & \rm{ \ otherwise}.
                        \end{cases}$$
Observe that $\chi_\D(e,f)= \chi_{\D_0}(e,f)$ if and only if $\alpha(x)\beta(u)=1$.
                        Therefore,
$$\chi_\D(e,f)= \alpha(\lca_R(e,f)) \beta (\lca_L(e,f))\chi_{\D_0}(e,f).$$
Counting the number of crossings in a layout $\D$, we have  
$$\Cr(\D)=\sum_{\{e,f\}}    \frac{1+\chi_\D(e,f)}{2}    .$$
We have
\begin{eqnarray} \label{firstpart}
\Cr(\D) & = & \sum_{\{e,f\}}    \frac{1+ \alpha(\lca_R(e,f)) \beta (\lca_L(e,f))\chi_{\D_0}(e,f)            }{2}    \\
 &=& \frac{1}{2} \binom{n}{2} + \frac{1}{2} \sum_{x\in \inter(R)}\sum_{u\in \inter(L) } \alpha(x)\beta(u)         \sum_{\{e,f\}:x=  \lca_R(e,f) \atop u=\lca_L(e,f) } \chi_{\D_0}(e,f).  \label{secpart}
\end{eqnarray}
To justify the claim in Section~\ref{intro} on the expected  number of crossings in a random layout of a labeled tanglegram, select randomly and independently the $\alpha$ and $\beta$ values to transform the fixed drawing $\D_0$
to transform it into the random drawing $\D$.
The displayed formula above implies $\EE[\Cr(\D)]=\frac{1}{2} \binom{n}{2}$.

\section{tools} \label{tools}
 
For a rooted binary  tree $T$, let $L(T)$ denote the set of its leaves. 
and set  $A(T)$ be the set of internal vertices that have a leaf neighbor. 
For $x\in A(T)$, set $\psi(x)=1$ if the number of leaves that are descendants of $x$ is even, 
$0$ otherwise, and
let $\psi_T=\sum_{x\in V(T)}\psi(x)$. 
Set $h(1)=0$ and for $n\ge 2$ let $h(n)=\min_{T:|L(T)|=n}\psi_T$. A tree $T$ with $n$ leaves is called a {\em realizer},  if $\psi_T=h(n)$. 

\begin{claim} \label{1vsodd}
For any $n\ge 2$, $h(n)=\lfloor\frac{n}{4}\rfloor+1$. In words, in any rooted binary tree with $n$ leaves, at least $\lfloor\frac{n}{4}\rfloor+1$ vertices have a leaf neighbor and an even number of leaf descendants.
\end{claim}

\begin{proof}
Note that $h(1)=0=\lfloor\frac{1}{4}\rfloor$. Let $n\ge 2$ and write $n=4q+r$ where $q=\lfloor\frac{n}{4}\rfloor$ and $0\le r<q$. We will show that 
 $h(n)=q+1$ by induction on $q$.

If $q=0$, then $n=r$. As $n\ge 2$, $r\in\{2,3\}$.
In these cases there is only one rooted binary tree on $n$ vertices.
Clearly, $h(2)=1$ and $h(3)=1$, so the claim is true.

Let $q>0$ and assume that the statement is true for all trees with $n^{\prime}$ leaves, where $2\le n^{\prime}<4q$.

Take a tree $T$ on $n$ vertices, and let $T_1,T_2$ be the 
subtrees rooted at the two neighbors of the root. Without loss of generality $1\leq k=|L(T_1)|\le|L(T_2)|=n-k\leq n-1$.
Set $q^{\prime}=\lfloor\frac{k}{4}\rfloor$, $r^{\prime}=k-4q^{\prime}$, $q^{\prime\prime}=\lfloor\frac{n-k}{4}\rfloor$ and $r''=(n-k)-4q^{\prime\prime}$.

When $0\le r^{\prime}\le r$ we get $r''=r-r'$, $q=q^{\prime}+q^{\prime\prime}$: 
 $\psi_T\ge\psi_{T_1}+\psi_{T_2}\ge q^{\prime}+(q-q^{\prime})+1=q+1$.  
The first inequality is an equality iff ($n$ is odd or ($n$ is even and $k\ne 1$)), and
 the second inequality is an equality iff $k=1$ and $T_2$ is a realizer.
Thus, for an odd $n$,
$h(n)\le q+1$  is obtained by choosing a tree $T$ such that $T_1$ is a single vertex and $T_2$ is a realizer with $n-1$ leaves.

When $r<r'\le 3$ we get $r''=4+r-r'$, $q^{\prime\prime}=q-q^{\prime}-1$, 
 $\psi_T\ge\psi_{T_1}+\psi_{T_2}\ge q^{\prime}+(q-q^{\prime}-1)+1=q$, the first inequality is an equality iff ($n$ is odd or ($n$ is even and $k\ne 1$)), the second inequality is an equality iff $k=1$ and $T_2$ is a realizer.
From $r<r'\le 3$ we get that if $n$ is odd, then $r=1$ and consequently $k\ge r'>1$; so for odd $n$ this gives that $h(n)\ge q+1$, and by the remark at the end of the previous paragraph, $h(n)=q+1$ for odd $n$.
If $n$ is even then $r\in\{0,2\}$. When $r=0$ then equality can not hold in both places, so $h(4q)\ge q+1$. Chosing $T$ with $4q$ leaves such that 
$T_1$ is single vertes and $T_2$ is a realizer with $4q-1-1$ leaves,
we get $\psi(T)=q+1$, so $h(4q)=q+1$.
When $r=2$, $r'=3$, consequently $k>1$. This gives $h(4q+2)\ge q+1$. Let $T_1$ be the tree on $3$ leaves and $T_2$ be a realizer
onb $4q-1$ leaves, then $\psi(T)=q+1$, so $h(4q+2)=q+1$.
\end{proof}

\section{unbalancing lights} \label{unbalance}
Alon and Spencer \cite{probmeth} contains the following theorem:
\begin{theorem} \label{as}
Let $a_{ij}=\pm 1$ for $1\leq i,j\leq n$. Then there exists $x_i,y_j=\pm 1$, $1\leq i,j\leq n$, so that 
\begin{equation} \label{pm1quote}
\sum_{i=1}^n \sum_{j=1}^n a_{ij}x_iy_j \geq \Bigl( \sqrt{\frac{2}{\pi}}+o(1)\Bigl)n^{3/2}.
\end{equation}
\end{theorem}

Alon and Spencer \cite{probmeth} gives an amusing interpretation of Theorem~\ref{as}, which explains the title  of this section:
``Let an $n\times n$ array of lights be given, each either on ($a_{ij}=+1$)  or off ($a_{ij}=-1$). Suppose for each row and each column there is a switch so that 
if the switch is pulled ($x_{i}=-1$  for  row $i$ and   $y_{j}=-1$ for column $j$) all of the lights in that line are `switched`: on to off and off to on. Then for any initial
configuration it is possible to perform switches so that the number of lights on minus the number of lights off is at least $\Bigl( \sqrt{\frac{2}{\pi}}+o(1)\Bigl)n^{3/2}$."

Clearly, redefining all $y_j$ to their negative, one obtains from (\ref{pm1quote}) 
\begin{equation} \label{pmquote}
\sum_{i=1}^n \sum_{j=1}^n a_{ij}x_iy_j \leq -\Bigl( \sqrt{\frac{2}{\pi}}+o(1)\Bigl)n^{3/2}.
\end{equation}
If we want to find a drawing $\D$ of the tanglegram $\T$, where the number of crossings   $\Cr(\D) $  in  $\D$ is way below  $\frac{1}{2} \binom{n}{2} $ in (\ref{firstpart})--(\ref{secpart}), then
the formulation (\ref{pmquote}) of Theorem~\ref{as} is absolutely relevant---except that instead of $a_{ij}=\pm 1$, we have to deal with
$$a_{xu} =  \sum_{\{e,f\}:x=  \lca_R(e,f) \atop u=\lca_L(e,f) } \chi_{\D_0}(e,f),$$
 in
\begin{equation}
\sum_{x\in \inter(R)}\sum_{u\in \inter(L) } \alpha(x)\beta(u)         \sum_{\{e,f\}:x=  \lca_R(e,f) \atop u=\lca_L(e,f) } \chi_{\D_0}(e,f).  \label{oursum}
\end{equation}
The difficulty is that now $a_{xu}$ may take other values than $\pm 1$, in fact, it is difficult to find many non-zero $a_{xu}$ terms. Therefore we were unable to utilize the probabilistic method.

Assume that $x\in R$ satisfies the Claim~\ref{1vsodd}, i.e., it has a leaf neighbor and has an even number of leaf descendants. 
We are going to call such $x\in {\rm int}(R)$ internal vertices as {\it special vertices}, and denote the set of them with $S$.
 Let $e$ be the leaf neighbor of the special vertex $x$ and let
$f_1,f_2,...,f_{2k+1}$ be the matching edges at  further leaf descendants of $x$. We have  
$$\sum_{u\in {\rm int}(L)} a_{xu}=\sum_{i=1}^{2k+1} \chi_{\D_0}(e,f_i) \not\equiv 0 \hbox{\  \ mod } 2.$$

Let us be given an arbitrary tanglegram $\T$ of size $n$. Without loss of generality, the layout $\D_0$ realizes the crossing number of
$\T$. Then for every $x\in S$, $\sum_{u\in {\rm int}(L)} a_{xu}<0$, as this sum is non-zero, and if it was 
positive, a switch in $x$ would yield a layout with strictly smaller number of crossings.

Consider now the following layout $\D_1$: switch in all  $u\in {\rm int}(L)$. 
It is easy to see that
$$\Cr(\D_1)=\binom{n}{2} -\Cr(\D_0).$$
Now switch in layout $\D_1$ at every $x\in S$ vertex to obtain layout $\D_2$:
\begin{eqnarray*}
\Cr(\D_2)&=&\Cr(\D_1)+  2\sum_{x\in S}    \sum_{u\in {\rm int}(L)} a_{xu}\\
&=&     \binom{n}{2} -\Cr(\D_0)+  2\sum_{x\in S}    \sum_{u\in {\rm int}(L)} a_{xu}\\
&\leq & \binom{n}{2} -\Cr(\D_0)-2|S|.
\end{eqnarray*}

Hence 
$$\Cr(\T)= \Cr(\D_0)\leq \frac{1}{2}\left(\Cr(\D_0)+\Cr(\D_2)\right)\leq \frac{1}{2}\binom{n}{2}- |S|\leq
\frac{1}{2}\binom{n}{2}- \frac{n}{4}
$$
by Claim~\ref{1vsodd}.

\end{document}